\documentclass[12pt]{article}
\usepackage{amsmath,amscd,amsthm,a4,amssymb}

\def\dual{\,^{^{\complement}}\!}
\def\Rn{{\mathbb{R}^n}}

\def\i{\infty}
\def\a {\alpha}

 \newtheorem{thm}{Theorem}[section]
 
 \newtheorem{lem}[thm]{Lemma}
 
 \theoremstyle{definition}
 \newtheorem{defn}[thm]{Definition}
 \theoremstyle{remark}
 \newtheorem{rem}[thm]{Remark}
 
 \numberwithin{equation}{section}

\usepackage{amssymb}
\usepackage{color}

\def\Rn{{\mathbb{R}^n}}
\def\a {\alpha}
\def\i{\infty}

\def\L1loc{L_{\Phi}^{\rm loc}(\Rn)}

\def\dual{\,^{^{\complement}}\!}

\newcommand{\ess}{\mathop{\rm ess \; sup}\limits}
\newcommand{\es}{\mathop{\rm ess \; inf}\limits}

\begin{document}

\begin{center}
\LARGE A characterization for fractional integral and its commutators in Orlicz and generalized Orlicz-Morrey spaces on spaces of homogeneous type
\end{center}

\

\vspace{3mm}
\centerline{\large Vagif S. Guliyev}
\centerline{\it Department of Mathematics, Ahi Evran University, Kirsehir, Turkey}
\centerline{\it Institute of Mathematics and Mechanics, Baku, Azerbaijan}
\centerline{vagif@guliyev.com}

\vspace{3mm}

\centerline{\large Fatih Deringoz}
\centerline{\it Department of Mathematics, Ahi Evran University, Kirsehir, Turkey}
\centerline{deringoz@hotmail.com}

\

\begin{abstract}
In this paper, we investigate the boundedness of maximal operator and its commutators in generalized Orlicz-Morrey spaces on the spaces of homogeneous type.
As an application of this boundedness, we give necessary and sufficient condition for the Adams type boundedness of fractional integral and its commutators in these spaces. We also discuss criteria for the boundedness of these operators in Orlicz spaces.
\end{abstract}

\

\



\

\section{Introduction}

In the 1970's, in order to extend the theory of Calder\'{o}n-Zygmund singular
integrals to a more general setting, R. Coifman and G. Weiss introduced
certain topological measure spaces which are equipped with a metric which
is compatible with the given measure in a sense. These spaces are called spaces of homogeneous type. In this
work, we find necessary and sufficient conditions for the boundedness of fractional integral and its commutators
in Orlicz and generalized Orlicz-Morrey spaces on spaces of homogeneous type.

As a generalization of $L_p(\Rn)$, the Orlicz spaces were introduced by Birnbaum-Orlicz in \cite{Birnbaum-Orlicz} and Orlicz in \cite{Orlicz1}, since then, the theory of the Orlicz spaces themselves has been well developed and the spaces have been widely used in probability, statistics, potential theory, partial differential equations, as well as harmonic analysis and some other fields of analysis. They have been thoroughly investigated, and two excellent monographs \cite{KrasnRut} and \cite{RaoRen} are available on this subject. Also \cite{BenSharp} provides a good overview on the subject.

The spaces $M_{p,\varphi}(\Rn)$ defined by the norm
$$
\left\| f\right\|_{M_{p,\varphi}}:
= \sup_{x \in \Rn, \; r>0 } \varphi(r)^{-1}\,|B(x,r)|^{-\frac{1}{p}} \|f\|_{L_p (B(x,r))}
$$
with a  function $\varphi$ positive and measurable on $\Rn\times (0,\infty )$ are known as generalized Morrey spaces. For certain functions $\varphi$, the spaces $M_{p,\varphi}(\Rn)$ reduce to some classical spaces. For instance, if $\varphi(r)=r^{\frac{\lambda-n}{p}}$, where $0\leq \lambda\leq n$, then $M_{p,\varphi}$ is the classical Morrey space $M_{p,\lambda}$.

The classical result by Hardy-Littlewood-Sobolev states that if $1<p<q<\i$, then the fractional integral (also known as Riesz potential) $I_{\a}$ ($0 < \a <n$) is bounded from $L_{p}(\Rn)$ to $L_{q}(\Rn)$ if and only if $\a=n\left(\frac{1}{p}-\frac{1}{q}\right)$
The Hardy-Littlewood-Sobolev theorem is an important result in the fractional
integral theory and the potential theory. Later then, this result has been extended from Lebesgue spaces to various function spaces.

Around the 1970's, the Hardy-Littlewood-Sobolev inequality is extended from Lebesgue spaces to Morrey spaces. As stated in \cite{Peetre}, Spanne proved the following result.

\begin{thm} (Spanne, but published by Peetre \cite{Peetre}) \label{Peetre1}
Let $0<\alpha<n$, $1<p<\frac{n}{\alpha}$, $0<\lambda<n-\alpha p$. Moreover, let
$\frac{1}{p}-\frac{1}{q}=\frac{\alpha}{n}$ and $\frac{\lambda}{p}=\frac{\mu}{q}$.
Then the operator $I_{\alpha}$ is bounded from $M_{p,\lambda}(\Rn)$
to $M_{q,\mu}(\Rn)$.
\end{thm}

Later on, a stronger result was obtained by Adams \cite{Adams}, and reproved by Chiarenza and Frasca \cite{ChiFra}.

\begin{thm} (Adams \cite{Adams}) \label{Adams1}
Let $0<\alpha<n$, $1<p<\frac{n}{\alpha}$, $0<\lambda<n-\alpha p$ and
$\frac{1}{p}-\frac{1}{q}=\frac{\alpha}{n-\lambda}$.
Then the operator $I_{\alpha}$ is bounded from $M_{p,\lambda}(\Rn)$
to $M_{q,\lambda}(\Rn)$.
\end{thm}

For the boundedness of $I_{\alpha}$ on generalized Morrey spaces see \cite{GulJIA,GULAKShIEOT2012,Nakai94,SugTan} and references therein. The fractional integral in Orlicz spaces was studied in \cite{Cianchi1,Nakai,O'Neil,Torch1}. For more details we refer to survey paper \cite{Nakaisurv}.

Commutators of classical operators of harmonic analysis play an important  role in various topics of analysis and PDE, see for instance \cite{Cald1,Chanillo,CRW,CLMS},  where in particular in \cite{Chanillo} it was shown that the commutator $[b,I_{\a}]$ is bounded from $L^p(\Rn)$ to $L^q(\Rn)$ for $1< p < \frac{n}{\a}$, $\frac{1}{q}=\frac{1}{p}-\frac{\a}{n}$ and $b\in BMO(\Rn)$.

In order to extend the traditional Euclidean space to build a general underlying structure for the real harmonic analysis, the notion of spaces of
homogeneous type was introduced by Coifman and Weiss \cite{CoifWeiss}.

Let $X=(X,d,\mu)$ be a space of homogeneous type, i.e. $X$ is a topological space endowed with a quasi-distance $d$ and a positive measure $\mu$ such that
$$d(x,y)\geq0 \text{   and  } d(x,y)=0 \text{  if and only if  } x=y,$$
$$d(x,y)=d(y,x),$$
\begin{equation}\label{triinq}
d(x,y)\leq K_1(d(x,z)+d(z,y)),
\end{equation}
the balls $B(x,r)=\{y\in X: d(x,y)<r\},\,r>0$, form a basis of neighborhoods of the point $x$, $\mu$ is defined on a $\sigma$-algebra of subsets of $X$ which contains the balls, and
\begin{equation}\label{doubl1}
0<\mu(B(x,2r))\leq K_2\,\mu(B(x,r))<\i,
\end{equation}
where $K_i\geq1\,(i=1,2)$ are constants independent of $x, y, z \in X$ and $r > 0$. As usual, the dilation of a ball $B = B(x,r)$ will be denoted by $\lambda B = B(x, \lambda r)$ for every $\lambda > 0$.

Note that \eqref{doubl1} implies that
\begin{equation}\label{doubl2}
\mu(\lambda B)\leq C(\mu,\lambda)\,\mu(B),
\end{equation}
for all $\lambda\geq 1$.

In the sequel, we always assume that
$\mu(X)=\i$, the space of compactly supported continuous function is dense in $L_1(X,\mu)$ and that $X$ is $Q$-homogeneous $(Q > 0)$, i.e.
\begin{equation}\label{Qhomogeneous}
K_3^{-1}r^Q\leq \mu(B(x,r)) \leq K_3 r^Q,
\end{equation}
where $K_3\geq 1$ is a constant independent of $x$ and $r$. The $n$-dimensional Euclidean space $\Rn$ is $n$-homogeneous.

In proving the boundedness of the fractional integral operators on various spaces, some researchers find that the translation invariance and the doubling properties of the Lebesgue measure play an important role. This is also true in studying other operators such as maximal operators and various types of singular integral operators. Thus, inspired by this fact, they studied the operators in the homogeneous setting. We refer to \cite{DengHan,GenGogKokKr,GulMusHom,Nakai-Hom,Nakai-HomFS} and references therein.

The authors introduced generalized Orlicz-Morrey spaces in \cite{DGS} to investigate the boundedness of maximal and singular operators. Generalized Orlicz-Morrey spaces unify Orlicz and generalized Morrey spaces. Also, in \cite{DGSJIA} the authors extended the Adams type boundedness of Riesz potential and its commutators to the generalized Orlicz-Morrey spaces on the $n$-dimensional Euclidean space $\Rn$. Moreover, the authors find criteria for the boundedness of Riesz potential and its commutators on Orlicz spaces on the $n$-dimensional Euclidean space $\Rn$ in \cite{GulDerHas}.
The purpose of this paper is to extend these results to the spaces of homogeneous type.

Before describing the characterization for fractional integral and its commutators in Orlicz and generalized Orlicz-Morrey spaces on spaces of
homogeneous type, we give several examples of spaces of homogeneous type (\cite{CoifWeiss, CoifWeiss2,DengHan,GenGogKokKr}).

$(1)~~~ X=\Rn, ~ \rho(x,y)=|x-y|=\Big(\sum\limits_{j=1}^n (x_j-y_j)^2 \Big)^{\frac{1}{2}}$ and $\mu$ equals Lebesgue
measure.

$(2)~~~ X=\Rn, ~ \rho(x,y)=\sum\limits_{j=1}^n (x_j-y_j)^{\alpha_j}$, where $\alpha_1,\alpha_2,\ldots,\alpha_n$ are positive
numbers, not necessarily equal, and ì equals Lebesgue measure (this distance is called nonisotropic).

$(3)~~~ X=[0,1), ~ \rho(x,y)$ is the length of the smallest dyadic interval containing $x$ and $y$, and $\mu$ is Lebesgue measure.

$(4)~~~$ Any $C^{\infty}$ compact Riemannian manifold with the Riemannian metric and volume.

$(5)~~~$ Let $G$ be a nilpotent Lie group with a left-invariant Riemannian metric and $\mu$ is the induced measure.

$(6)~~~$ When $X$ is the boundary of a smooth and bounded pseudo-convex domain in $\mathbb{C}^n$ one can introduce a
nonisotropic quasi-distance that is related to the complex structure in such a way that we obtain a space of homogeneous
type by using Lebesgue surface measure. For example, if $X$ is the surface of the unit sphere
$$
\sigma_{2n-1}=\Big\{ z \in \mathbb{C}^n : z \cdot \overline{z} = \sum\limits_{j=1}^n z_j \overline{z_j} = 1 \Big\},
$$
the nonisotropic distance is given by $d(z,w)=|1-z \cdot \overline{w}|^{\frac{1}{2}}$.

By $A\lesssim B$ we mean that $A\le CB$ with some positive constant $C$ independent
of appropriate quantities. If $A\lesssim B$ and $B\lesssim A$, we write $A\thickapprox B$ and say that $A$
and $B$ are equivalent.

\section{Preliminaries}

The Morrey spaces and weak Morrey spaces on spaces of homogeneous type are defined as follows.
\begin{defn} Let $1\le p < \infty$ and $0 \le \lambda \le  Q,$
\begin{equation*}
M_{p,\lambda}(X)=\left\{f\in L_{p}^{\rm loc}(X): \left\| f\right\|_{M_{p,\lambda}}: = \sup_{x\in X, \; r>0 } r^{-\frac{\lambda}{p}} \|f\|_{L_{p}(B(x,r))}< \infty\right\},
\end{equation*}
\begin{equation*}
WM_{p,\lambda}(X)=\left\{f\in L_{p}^{\rm loc}(X): \left\| f\right\|_{WM_{p,\lambda}}: = \sup_{x\in X, \; r>0 } r^{-\frac{\lambda}{p}} \|f\|_{WL_{p}(B(x,r))}< \infty\right\},
\end{equation*}
where
$$\|f\|_{L_{p}(B(x,r))}=\left(\int_{B(x,r)}|f(y)|^p d\mu(y)\right)^{\frac{1}{p}}$$
and $WL_{p}(B(x,r))$ denotes the weak $L_{p}$-space of measurable functions $f$ for which
$$ \left\| f\right\|_{WL_{p}(B(x,r))}=\sup_{\tau>0}\tau\mu\big(\{y\in B(x,r): |f(y)|>\tau\}\big)^{\frac{1}{p}}.$$
\end{defn}

We recall the definition of Young functions.

\begin{defn}\label{def2} A function $\Phi : [0,\infty) \rightarrow [0,\infty]$ is called a Young function if $\Phi$ is convex, left-continuous, $\lim\limits_{r\rightarrow +0} \Phi(r) = \Phi(0) = 0$ and $\lim\limits_{r\rightarrow \infty} \Phi(r) = \infty$.
\end{defn}

From the convexity and $\Phi(0) = 0$ it follows that any Young function is increasing.
If there exists $s \in (0,\infty)$ such that $\Phi(s) = \infty$, then $\Phi(r) = \infty$ for $r \geq s$.

Let $\mathcal{Y}$ be the set of all Young functions $\Phi$ such that
\begin{equation*}
0<\Phi(r)<\infty\qquad \text{for} \qquad 0<r<\infty
\end{equation*}
If $\Phi \in \mathcal{Y}$, then $\Phi$ is absolutely continuous on every closed interval in $[0,\infty)$
and bijective from $[0,\infty)$ to itself.

The Orlicz spaces and weak Orlicz spaces on spaces of homogeneous type are defined as follows.
\begin{defn} For a Young function $\Phi$,
$$
L_{\Phi}(X)=\left\{f\in L_1^{\rm loc}(X): \int_{X}\Phi(\epsilon|f(x)|)d\mu(x)<\infty
 \text{ for some $\epsilon>0$  }\right\},
$$
$$
\|f\|_{L_{\Phi}}=\inf\left\{\lambda>0:\int_{X}\Phi\Big(\frac{|f(x)|}{\lambda}\Big)d\mu(x)\leq 1\right\},
$$
$$
WL_{\Phi}(X):=\left\{f\in L_{\rm loc}^{1}(X):\sup_{r>0}\Phi(r)m\Big(r,\epsilon f\Big)<\infty  \text{ for some $\epsilon>0$  }\right\},
$$
$$
\Vert f\Vert_{WL_{\Phi}}=\inf\left\{\lambda>0\ :\ \sup_{t>0}\Phi(\frac{t}{\lambda})d_{f}(t)\ \leq 1\right\},
$$
where $d_{f}(t)=|\{x\in\Rn: |f(x)|>t\}|$.
\end{defn}
We note that,
\begin{equation}\label{normofcharac}
\|\chi_{_B}\|_{WL_{\Phi}} = \|\chi_{_B}\|_{L_{\Phi}} = \frac{1}{\Phi^{-1}\left(\mu(B)^{-1}\right)},
\end{equation}
where $B$ is a $\mu$-measurable set in $X$ with $\mu(B)<\i$ and $\chi_{_B}$ is the characteristic function of $B$, that
\begin{equation}\label{orlpr}
\int _{X}\Phi\Big(\frac{|f(x)|}{\|f\|_{L_{\Phi}}}\Big)d\mu(x)\leq 1
\end{equation}
and that
\begin{equation}\label{worlpr}
\sup_{t>0}\Phi(\frac{t}{\Vert f\Vert_{WL_{\Phi}}})d_{f}(t)\ \leq 1.
\end{equation}


For a Young function $\Phi$ and  $0 \leq s \leq \infty$, let
$$\Phi^{-1}(s)=\inf\{r\geq 0: \Phi(r)>s\}\qquad (\inf\emptyset=\infty).$$
If $\Phi \in \mathcal{Y}$, then $\Phi^{-1}$ is the usual inverse function of $\Phi$. We note that
\begin{equation*}
\Phi(\Phi^{-1}(r))\leq r \leq \Phi^{-1}(\Phi(r)) \quad \text{ for } 0\leq r<\infty.
\end{equation*}

A Young function $\Phi$ is said to satisfy the $\Delta_2$-condition, denoted by  $\Phi \in \Delta_2$, if
$$
\Phi(2r)\le k\Phi(r) \text{    for } r>0
$$
for some $k>1$. If $\Phi \in \Delta_2$, then $\Phi \in \mathcal{Y}$.

A Young function $\Phi$ is said to satisfy the $\nabla_2$-condition, denoted also by  $\Phi \in \nabla_2$, if
$$\Phi(r)\leq \frac{1}{2k}\Phi(kr),\qquad r\geq 0,$$
for some $k>1$.

For a Young function $\Phi$, the complementary function $\widetilde{\Phi}(r)$ is defined by
\begin{equation*}
\widetilde{\Phi}(r)=\left\{
\begin{array}{ccc}
\sup\{rs-\Phi(s): s\in [0,\infty)\}
& , & r\in [0,\infty), \\
\infty&,& r=\infty.
\end{array}
\right.
\end{equation*}
The complementary function  $\widetilde{\Phi}$ is also a Young function and $\widetilde{\widetilde{\Phi}}=\Phi$.
If $\Phi(r)=r$, then $\widetilde{\Phi}(r)=0$ for $0\leq r \leq 1$ and $\widetilde{\Phi}(r)=\infty$
for $r>1$. If $1 < p < \infty$, $1/p+1/p^\prime= 1$ and $\Phi(r) =
r^p/p$, then $\widetilde{\Phi}(r) = r^{p^\prime}/p^\prime$. If $\Phi(r) = e^r-r-1$, then $\widetilde{\Phi}(r) = (1+r) \log(1+r)-r$.
Note that $\Phi \in \nabla_2$ if and only if $\widetilde{\Phi} \in \Delta_2$. It is known that
\begin{equation}\label{2.3}
r\leq \Phi^{-1}(r)\widetilde{\Phi}^{-1}(r)\leq 2r \qquad \text{for } r\geq 0.
\end{equation}

Note that by the convexity of $\Phi$ and concavity of $\Phi^{-1}$ we have the following properties
\begin{equation}\label{convconc}
\left\{
\begin{array}{ccc}
\Phi(\alpha t)\leq \alpha \Phi(t),
& \text{  if } &0\le\a\le1 \\
\Phi(\a t)\geq \a \Phi(t),&\text{  if }& \a>1
\end{array}
\right.
\text{ and }
\left\{
\begin{array}{ccc}
\Phi^{-1}(\alpha t)\geq \alpha \Phi^{-1}(t),
& \text{  if } &0\le\a\le1 \\
\Phi^{-1}(\a t)\leq \a \Phi^{-1}(t),&\text{  if }& \a>1.
\end{array}
\right.
\end{equation}

The following analogue of the H\"older inequality is known,
\begin{equation}\label{HoSHTOrl}
\int_X|f(x)g(x)|d\mu(x) \leq 2 \|f\|_{L_{\Phi}} \|g\|_{L_{\widetilde{\Phi}}}.
\end{equation}

In the next sections where we prove our main estimates, we use the following lemma, which follows from H\"older inequality, \eqref{normofcharac} and \eqref{2.3}.
\begin{lem}\label{lemHold}
For a Young function $\Phi$ and $B=B(x,r)$, the following inequality is valid
$$\|f\|_{L_{1}(B)} \leq 2 \mu(B) \Phi^{-1}\left(\mu(B)^{-1}\right) \|f\|_{L_{\Phi}(B)},$$
where $\|f\|_{L_{\Phi}(B)}=\|f\chi_{_B}\|_{L_{\Phi}}$.
\end{lem}

\section{Generalized Orlicz-Morrey spaces}

The generalized Orlicz-Morrey spaces and the weak generalized Orlicz-Morrey spaces on spaces of homogeneous type are defined as follows.
\begin{defn}
Let $(X,d,\mu)$ be $Q-$homogeneous, $\varphi(r)$ be a positive measurable function on $(0,\infty)$ and $\Phi$ any Young function.
We denote by $M_{\Phi,\varphi}(X)$ the generalized Orlicz-Morrey space, the space of all
functions $f\in L_{\Phi}^{\rm loc}(X)$ with finite quasinorm
$$
\|f\|_{M_{\Phi,\varphi}} \equiv \|f\|_{M_{\Phi,\varphi}(\Rn)} = \sup\limits_{x\in X, r>0}
\varphi(r)^{-1} \Phi^{-1}(\mu(B(x,r))^{-1}) \|f\|_{L_{\Phi}(B(x,r))},
$$
where $L_{\Phi}^{\rm loc}(X)$ is defined as the set of all functions $f$ such that  $f\chi_{_B}\in L_{\Phi}(X)$ for all balls $B \subset X$.

Also by $WM_{\Phi,\varphi}(X)$ we denote the weak generalized Orlicz-Morrey space of all functions $f\in WL_{\Phi}^{\rm loc}(X)$ for which
$$
\|f\|_{WM_{\Phi,\varphi}} \equiv \|f\|_{WM_{\Phi,\varphi}(\Rn)} = \sup\limits_{x\in X, r>0} \varphi(r)^{-1} \Phi^{-1}(\mu(B(x,r))^{-1}) \|f\|_{WL_{\Phi}(B(x,r))} < \infty,
$$
where $WL_{\Phi}^{\rm loc}(X)$ is defined as the set of all functions $f$ such that  $f\chi_{_B}\in WL_{\Phi}(X)$ for all balls $B \subset X$.
\end{defn}

\begin{rem}
Thanks to \eqref{Qhomogeneous} and \eqref{convconc} we have
\begin{equation*}
\Phi^{-1}(\mu(B(x,r))^{-1})\approx \Phi^{-1}(r^{-Q}).
\end{equation*}
Therefore we can also write
$$
\|f\|_{M_{\Phi,\varphi}} \equiv \sup\limits_{x\in X, r>0}
\varphi(r)^{-1} \Phi^{-1}(r^{-Q}) \|f\|_{L_{\Phi}(B(x,r))},
$$
and
$$
\|f\|_{WM_{\Phi,\varphi}} \equiv \sup\limits_{x\in X, r>0} \varphi(r)^{-1} \Phi^{-1}(r^{-Q}) \|f\|_{WL_{\Phi}(B(x,r))},
$$
respectively.
\end{rem}
According to this definition, we recover the generalized Morrey space $M_{p,\varphi}(X)$ and weak generalized Morrey space $WM_{p,\varphi}(X)$ under the choice $\Phi(r)=r^{p},\,1\le p<\i$. If $\Phi(r)=r^{p},\,1\le p<\i$ and $\varphi(r)=r^{\frac{\lambda-Q}{p}},\,0\le \lambda \le Q$, then $M_{\Phi,\varphi}(X)$ and $WM_{\Phi,\varphi}(X)$ coincide with $M_{p,\lambda}(X)$ and $WM_{p,\lambda}(X)$, respectively and if $\varphi(r)=\Phi^{-1}(r^{-Q})$, then $M_{\Phi,\varphi}(X)$ and $WM_{\Phi,\varphi}(X)$ coincide with the $L_{\Phi}(X)$ and $WL_{\Phi}(X)$, respectively.

A function $\varphi:(0,\infty) \to (0,\infty)$ is said to be almost increasing (resp.
almost decreasing) if there exists a constant $C > 0$ such that
$$
\varphi(r)\leq C \varphi(s)\qquad (\text{resp. }\varphi(r)\geq C \varphi(s))\quad \text{for  } r\leq s.
$$
For a Young function $\Phi$, we denote by ${\mathcal{G}}_{\Phi}$ the set of all almost decreasing functions $\varphi:(0,\infty) \to (0,\infty)$
such that $t\in(0,\infty) \mapsto \frac{\varphi(t)}{\Phi^{-1}(t^{-Q})}$ is almost increasing.

\begin{lem}\label{charOrlMor}
Let $B_0:=B(x_0,r_0)$. If $\varphi\in{\mathcal{G}}_{\Phi}$, then there exist $C>0$ such that
$$
\frac{1}{\varphi(r_0)}\leq \|\chi_{B_0}\|_{M_{\Phi,\varphi}}\leq \frac{C}{\varphi(r_0)}.
$$
\end{lem}
\begin{proof}
Let $B=B(x,r)$ denote an arbitrary ball in $X$. By the definition and \eqref{normofcharac}, it is easy to see that
\begin{align*}
\|\chi_{B_0}\|_{M_{\Phi,\varphi}}&=\sup\limits_{x\in X, r>0}\varphi(r)^{-1}\Phi^{-1}(\mu(B)^{-1})\frac{1}{\Phi^{-1}(\mu(B\cap B_0)^{-1})}
\end{align*}
\begin{align*}
&\geq \varphi(r_0)^{-1}\Phi^{-1}(\mu(B_0)^{-1})\frac{1}{\Phi^{-1}(\mu(B_0\cap B_0)^{-1})} = \frac{1}{\varphi(r_0)}.
\end{align*}

Now if $r\leq r_0$, then $\varphi(r_0)\leq C\varphi(r)$ and
\begin{align*}
\varphi(r)^{-1}\Phi^{-1}(\mu(B)^{-1})\|\chi_{B_0}\|_{L_{\Phi}(B)}\leq\frac{1}{\varphi(r)}\leq \frac{C}{\varphi(r_0)}.
\end{align*}

On the other hand if $r\geq r_0$, then $\frac{\varphi(r_0)}{\Phi^{-1}(\mu(B_0)^{-1})}\leq C\frac{\varphi(r)}{\Phi^{-1}(\mu(B)^{-1})}$ and
\begin{align*}
\varphi(r)^{-1}\Phi^{-1}(\mu(B)^{-1})\|\chi_{B_0}\|_{L_{\Phi}(B)}\leq \frac{C}{\varphi(r_0)}.
\end{align*}
This completes the proof.
\end{proof}

\section{Maximal operator and its commutators in generalized Orlicz-Morrey spaces}

Let $Mf(x)$ be the maximal function, i.e.
$$
Mf(x)=\sup\limits_{r>0}\frac{1}{\mu(B(x,r))} \int_{B(x,r)}|f(y)|d\mu(y).
$$

The known boundedness statement for $M$ in Orlicz spaces on spaces of homogeneous
type runs as follows.

\begin{thm}\cite{GenGogKokKr}\label{maxorl}
Let $\Phi \in \mathcal{Y}$. Then $M$ is bounded from $L_{\Phi}(X)$ to $WL_{\Phi}(X)$. Moreover, if $\Phi\in\nabla_2$, then $M$ is bounded from $L_{\Phi}(X)$ to $L_{\Phi}(X)$.
\end{thm}

\begin{lem}\label{lem4.2.}
Let $\Phi \in \mathcal{Y}$, $f\in L_\Phi^{\rm loc}(X)$ and  $B=B(x,r)$. Then
\begin{equation}\label{4.5}
\|M f\|_{L_{\Phi}(B)} \lesssim  \frac{1}{\Phi^{-1}\big(r^{-Q}\big)} \, \sup_{t>r} \Phi^{-1}\big(t^{-Q}\big) \, \|f\|_{L_{\Phi}(B(x,t))}
\end{equation}
for any Young function  $\Phi \in \nabla_2$
and
\begin{align}\label{eq100WX}
\|M f\|_{WL_{\Phi}(B)} \lesssim \frac{1}{\Phi^{-1}\big(r^{-Q}\big)} \, \sup_{t>r} \Phi^{-1}\big(t^{-Q}\big) \, \|f\|_{L_{\Phi}(B(x,t))}
\end{align}
for any Young function $\Phi$.
\end{lem}

\begin{proof}
Let $\Phi \in \nabla_2$. We put $f=f_1+f_2$, where $f_1=f\chi_{B(x,2kr)}$ and $f_2=f\chi_{{\dual}B(x,2kr)}$, where $k$ is the constant from the triangle inequality \eqref{triinq}.

\emph{Estimation of $Mf_1$}: By Theorem \ref{maxorl} we have
$$
\|M f_1\|_{L_{\Phi}(B)}\leq \|M f_1\|_{L_{\Phi}(X)}\lesssim \|f_1\|_{L_{\Phi}(X)}=\|f\|_{L_{\Phi}(B(x,2kr))}.
$$
By using the monotonicity of the functions $\|f\|_{L_{\Phi}(B(x,t))}$, ${\Phi^{-1}\big(t\big)}$ with respect to $t$ and \eqref{convconc} we get,
\begin{equation}\label{dsaad}
\begin{split}
&\frac{1}{\Phi^{-1}\big(r^{-Q}\big)} \,  \sup_{t>2kr} \Phi^{-1}\big(t^{-Q}\big) \|f\|_{L_{\Phi}(B(x,t))}
\\
&\hskip+1cm\geq \frac{\|f\|_{L_{\Phi}(B(x,2kr))}}{\Phi^{-1}\big(r^{-Q}\big)} \,  \sup_{t>2kr} \Phi^{-1}\big(t^{-Q}\big)\gtrsim \|f\|_{L_{\Phi}(B(x,2kr))}.
\end{split}
\end{equation}
Consequently we have
\begin{equation}\label{mf1}
\|M f_1\|_{L_{\Phi}(B)} \lesssim \frac{1}{\Phi^{-1}\big(r^{-Q}\big)} \,  \sup_{t>r} \Phi^{-1}\big(t^{-Q}\big) \|f\|_{L_{\Phi}(B(x,t))}
\end{equation}

\emph{Estimation of $Mf_2$}: Let $y$ be an arbitrary point from $B$. If $B(y,t)\cap {\dual}(B(x,2kr))\neq\emptyset,$ then $t>r$. Indeed, if $z\in B(y,t)\cap  {\dual} (B(x,2kr)),$
then $t > d(y,z) \geq \frac{1}{k}d(x,z)-d(x,y)>2r-r=r$.

On the other hand, $B(y,t)\cap {\dual} (B(x,2kr))\subset B(x,2kt)$. Indeed, if  $z\in B(y,t)\cap {\dual} (B(x,2kr))$, then
we get $d(x,z)\leq kd(y,z)+kd(x,y)<kt+kr<2kt$.

Therefore,
\begin{equation*}
\begin{split}
M f_2(y) & = \sup_{t>0}\frac{1}{\mu(B(y,t))} \int_{B(y,t)\cap {{\dual}(B(x,2kr))}}|f(z)|d\mu(z)
\\
& \le \sup_{t>r}\frac{1}{\mu(B(y,t))} \int_{B(x,2kt)}|f(z)|d\mu(z)
\\
&\le   \sup_{t>r} \frac{C}{\mu(B(y,2kt))} \int_{B(x,2kt)}|f(z)|d\mu(z)
\\
&=   \sup_{t>2kr} \frac{C}{\mu(B(y,t))} \int_{B(x,t)}|f(z)|d\mu(z)
\end{split}
\end{equation*}
by the doubling condition \eqref{doubl2}.

Hence by Lemma \ref{lemHold} and \eqref{Qhomogeneous}
\begin{equation}\label{fadhg}
\begin{split}
M f_2(y)  \lesssim   \sup_{t>2kr} \frac{\mu(B(x,t))}{\mu(B(y,t))} \Phi^{-1}\big(\mu(B(x,t))^{-1}\big) \|f\|_{L_{\Phi}(B(x,t))} \lesssim \sup_{t>r} \Phi^{-1}\big(t^{-Q}\big) \|f\|_{L_{\Phi}(B(x,t))}
\end{split}
\end{equation}
Thus the function $Mf_2(y)$, with fixed $x$ and $r$, is dominated by the expression not depending on $y$. Then we integrate the obtained estimate for $Mf_2(y)$ in $y$ over $B$, we get
\begin{equation}\label{mf2}
\|M f_2\|_{L_{\Phi}(B)} \lesssim \frac{1}{\Phi^{-1}\big(r^{-Q}\big)} \,  \sup_{t>r} \Phi^{-1}\big(t^{-Q}\big) \|f\|_{L_{\Phi}(B(x,t))}
\end{equation}
Gathering the estimates \eqref{mf1} and \eqref{mf2} we arrive at \eqref{4.5}.

Let now $\Phi$ be an arbitrary  Young function. It is obvious that
\begin{gather*}
\|M f\|_{WL_{\Phi}(B)} \leq \|M f_1\|_{WL_{\Phi}(B)}+ \|M f_2\|_{WL_{\Phi}(B)}.
\end{gather*}
By the boundedness of the operator $M$ from $L_{\Phi}(X)$ to $WL_{\Phi}(X)$, provided by Theorem
\ref{maxorl},  we have
$$
\|M f_1\|_{WL_{\Phi}(B)} \lesssim \|f\|_{L_{\Phi}(B(x,2kr))}.
$$
By using \eqref{dsaad}, \eqref{fadhg} and \eqref{normofcharac} we arrive at \eqref{eq100WX}.
\end{proof}

\begin{thm}\label{thm4.4.max}
Let  $\Phi \in \mathcal{Y}$, the functions $\varphi_1,\varphi_2$ and $ \Phi$ satisfy the condition
\begin{equation}\label{bounmax}
\sup_{r<t<\infty} \Phi^{-1}\big(t^{-Q}\big) \es_{t<s<\i}\frac{\varphi_1(s)}{\Phi^{-1}\big(s^{-Q}\big)} \le C \, \varphi_2(r),
\end{equation}
where $C$ does not depend on $r$. Then the maximal operator $M$ is bounded from $M_{\Phi,\varphi_1}(X)$ to $WM_{\Phi,\varphi_2}(X)$ and
for $\Phi \in \nabla_2$, the operator $M$ is bounded from $M_{\Phi,\varphi_1}(X)$ to $M_{\Phi,\varphi_2}(X)$.
\end{thm}

\begin{proof}
Note that
\begin{align*}
\Big(\es\limits_{x\in A} f(x) \Big)^{-1}=\ess\limits_{x\in A}\frac{1}{f(x)}
\end{align*}
is true for any real-valued nonnegative function $f$ and measurable on $A$ and the fact that $\|f\|_{L_{\Phi}(B(x,t))}$ is a nondecreasing function of $t$
\begin{align*}
\frac{\|f\|_{L_{\Phi}(B(x,t))}}{\es_{0<t<s<\infty}\frac{\varphi_1(s)}{\Phi^{-1}\big(s^{-Q}\big)}}
& = \ess_{0<t<s<\infty} \frac{\Phi^{-1}\big(s^{-Q}\big)\|f\|_{L_{\Phi}(B(x,t))}}{\varphi_1(s)}
\\
& \le \sup\limits_{s>0, x\in X} \frac{\Phi^{-1}\big(s^{-Q}\big)\|f\|_{L_{\Phi}(B(x,s))}}{\varphi_1(s)}
 = \|f\|_{M_{\Phi,\varphi_1}}.
\end{align*}
Since $(\varphi_1,\varphi_2)$ and $\Phi$ satisfy the condition \eqref{bounmax},
\begin{align*}
& \sup\limits_{r<t<\infty}\|f\|_{L_{\Phi}(B(x,t))}
\Phi^{-1}\big(t^{-Q}\big)
\\
&\le \sup\limits_{r<t<\infty}\frac{\|f\|_{L_{\Phi}(B(x,t))}}{\es_{t<s<\infty} \frac{\varphi_1(s)}{\Phi^{-1}\big(s^{-Q}\big)}}
\es_{t<s<\infty} \frac{\varphi_1(s)}{\Phi^{-1}\big(s^{-Q}\big)}\Phi^{-1}\big(t^{-Q}\big)
\\
&\le C \|f\|_{M_{\Phi,\varphi_1}}\sup\limits_{r<t<\infty}
\Big(\es_{t<s<\infty}\frac{\varphi_1(s)}{\Phi^{-1}\big(s^{-Q}\big)}\Big) \, \Phi^{-1}\big(t^{-Q}\big)
\\
& \le C \varphi_2(r)\|f\|_{M_{\Phi,\varphi_1}}
\end{align*}
Then by \eqref{4.5}
\begin{equation*}
\begin{split}
\|M f\|_{M_{\Phi,\varphi_2}} & \lesssim \sup\limits_{x\in X, r>0}
\frac{1}{\varphi_2(r)}\,
\sup_{t>r} \Phi^{-1}\big(t^{-Q}\big) \|f\|_{L_{\Phi}(B(x,t))}
\\
& \lesssim \sup\limits_{x\in X, r>0}
\varphi_1(r)^{-1} \Phi^{-1}\big(r^{-Q}\big)\,  \|f\|_{L_{\Phi}(B(x,r))}
\\
& = \|f\|_{M_{\Phi,\varphi_1}}
\end{split}
\end{equation*}
The estimate $\|Mf\|_{WM_{\Phi,\varphi_2}}\lesssim\|f\|_{M_{\Phi,\varphi_1}}$ can be proved similarly by the help of local estimate \eqref{eq100WX}.
\end{proof}

The commutators generated by  $b\in  L^1_{\rm loc}(X)$ and the maximal operator $M$ is defined by
$$
M_{b}(f)(x)=\sup_{t>0}\mu(B(x,t))^{-1} \, \int _{B(x,t)}|b(x)-b(y)||f(y)|d\mu(y).
$$

We recall that the  space  $BMO(X)=\{ b\in L^1_{\rm loc}(X) ~ : ~ \| b \|_{\ast} < \infty  \}$  is defined by the seminorm
\begin{equation*}
\|b\|_\ast:=\sup_{x\in X, r>0}\frac{1}{\mu(B(x,r))} \int_{B(x,r)}|b(y)-b_{B(x,r)}|d\mu(y)<\infty,
\end{equation*}
where
$
b_{B(x,r)}=\frac{1}{\mu(B(x,r))} \int_{B(x,r)} b(y)d\mu(y).
$
We will need the following properties of BMO-functions:
\begin{equation} \label{lem2.4.}
\|b\|_\ast \thickapprox \sup_{x\in X, r>0}\left(\frac{1}{\mu(B(x,r))} \int_{B(x,r)}|b(y)-b_{B(x,r)}|^p d\mu(y)\right)^{\frac{1}{p}},
\end{equation}
where $1\leq p<\infty$, and
\begin{equation} \label{propBMO}
\left|b_{B(x,r)}-b_{B(x,t)}\right| \le C \|b\|_\ast \ln \frac{t}{r} \;\;\; \mbox{for} \;\;\; 0<2r<t,
\end{equation}
where $C$ does not depend on $b$, $x$, $r$ and $t$.

Next, we recall the notion of weights. Let $w$ be a locally integrable and positive
function on $X$. The function $w$ is said to be a Muckenhoupt $A_1$ weight if there exists a
positive constant $C$ such that for any ball $B$
$$\frac{1}{\mu(B)}\int_{B}w(x)d\mu(x) \le C \es_{x\in B}w(x).$$

\begin{lem}\label{rhi}\cite[Chapter 1]{GenGogKokKr}
Let $\omega\in A_1$, then the reverse H\"{o}lder inequality holds, that is, there exist $q > 1$ such that
$$
\left(\frac{1}{\mu(B)}\int_{B}w(x)^q d\mu(x)\right)^{\frac{1}{q}}\lesssim \frac{1}{\mu(B)}\int_{B}w(x)d\mu(x)
$$
for all balls $B$.
\end{lem}

\begin{lem}\label{keylem}
Let $\Phi$ be a Young function with $\Phi\in \Delta_2$. Then we have
$$
\frac{1}{2\mu(B)}\int_{B}|f(x)|d\mu(x)\le \Phi^{-1}\big(\mu(B)^{-1}\big)\left\|f\right\|_{L_{\Phi}(B)}\le C \left(\frac{1}{\mu(B)}\int_{B}|f(x)|^p d\mu(x)\right)^{\frac{1}{p}}
$$
for some $1<p<\infty$.
\end{lem}

\begin{proof}
The left-hand side inequality is just Lemma \ref{lemHold}.

Next we prove the right-hand side inequality. Our idea is based on \cite{IzukiSaw}. Take $g\in L_{\widetilde{\Phi}}$ with $\|g\|_{L_{\widetilde{\Phi}}}\le 1$. Note that $\widetilde{\Phi } \in \nabla_2$ since $\Phi \in \Delta_2$, therefore $M$ is bounded on $L_{\widetilde{\Phi}}(X)$ from Theorem \ref{maxorl}. Let $Q:=\|M\|_{L_{\widetilde{\Phi}}\to L_{\widetilde{\Phi}}}$ and define a function
$$
Rg(x):=\sum_{k=0}^{\infty}\frac{M^{k}g(x)}{(2Q)^k},
$$
where
\[ M^{k}g:= \left\{
\begin{array}{ll}
      |g| & k=0, \\
      Mg & k=1, \\
      M(M^{k-1}g) & k\geq 2. \\
\end{array}
\right. \]

For every $g\in L_{\widetilde{\Phi}}$ with $\|g\|_{L_{\widetilde{\Phi}}}\le 1$, the function $Rg$ satisfies the following properties:
\begin{itemize}
  \item $|g(x)|\le Rg(x)$ for almost every $x\in X$;
  \item $\|Rg\|_{L_{\widetilde{\Phi}}}\le 2\|g\|_{L_{\widetilde{\Phi}}}$
  \item $M(Rg)(x)\le 2QRg(x)$, that is, $Rg$ is a Muckenhoupt $A_1$ weight with the $A_1$ constant less than or equal to $2Q$.
\end{itemize}
By Lemma \ref{rhi}, there exist positive constants $q > 1$ and $C$ independent of $g$ such that for all balls $B$,
$$
\left(\frac{1}{\mu(B)}\int_{B}Rg(x)^q d\mu(x)\right)^{\frac{1}{q}}\le \frac{C}{\mu(B)} \int_{B}Rg(x) d\mu(x).
$$
By Theorem \ref{HoSHTOrl} and Lemma \ref{lemHold}, we obtain
\begin{align*}
  \|Rg\|_{L^q(B)}&=\mu(B)^{1/q}\left(\frac{1}{\mu(B)}\int_{B}Rg(x)^q d\mu(x)\right)^{\frac{1}{q}} \le \mu(B)^{1/q}\frac{C}{\mu(B)} \int_{B}Rg(x) d\mu(x)  \\
  & \le C\mu(B)^{-1/q^{\prime}}\frac{\|Rg\|_{L_{\widetilde{\Phi}}}}{\Phi^{-1}\big(\mu(B)^{-1}\big)}\le C \mu(B)^{-1/q^{\prime}} \frac{1}{\Phi^{-1}\big(\mu(B)^{-1}\big)}.
\end{align*}
Thus we have
\begin{align*}
   \int_{B}|f(x)g(x)|d\mu(x) & \le \int_{B}|f(x)|Rg(x)d\mu(x) \le \|f\|_{L_{q^\prime}(B)}\|Rg\|_{L_q(B)} \\
   & \le C \left(\frac{1}{\mu(B)}\int_{B}|f(x)|^{q^\prime} d\mu(x)\right)^{\frac{1}{q^\prime}}\frac{1}{\Phi^{-1}\big(\mu(B)^{-1}\big)}.
 \end{align*}
Since the Luxemburg-Nakano norm is equivalent to the Orlicz norm we get
\begin{align*}
\left\|f\right\|_{L_{\Phi}(B)} & \le \sup\left\{\left|\int_{B}f(x)g(x)d\mu(x)\right|: g\in L_{\widetilde{\Phi}},~~\|g\|_{L_{\widetilde{\Phi}}}\le 1\right\} \\
   & \le C \left(\frac{1}{\mu(B)}\int_{B}|f(x)|^{q^\prime} d\mu(x)\right)^{\frac{1}{q^\prime}}\frac{1}{\Phi^{-1}\big(\mu(B)^{-1}\big)}.
\end{align*}
Consequently, the right-hand side inequality follows with $p = q$.
\end{proof}

We have the following result from \eqref{lem2.4.} and Lemma \ref{keylem}.
\begin{lem}\label{Bmo-orlicz}
Let $b\in BMO(X)$ and $\Phi$ be a Young function with $\Phi\in\Delta_2$. Then
$$
\|b\|_\ast \thickapprox \sup_{x\in X, r>0}\Phi^{-1}\big(r^{-Q}\big)\left\|b(\cdot)-b_{B(x,r)}\right\|_{L_{\Phi}(B(x,r))}.
$$
\end{lem}

The known boundedness statements for the commutator operator $M_{b}$ on Orlicz spaces run as follows,
see \cite[Theorem 1.9 and Corollary 2.3]{FuYangYuan}. Note that in \cite{FuYangYuan} a more general case of multi-linear commutators was studied.

\begin{thm}\label{TinOrlicz}
Let $\Phi$ be a Young function with $\Phi\in\Delta_2\cap\nabla_2$ and $b\in BMO(X)$. Then $M_{b}$ is bounded on $L_{\Phi}(X)$ and the inequality
\begin{equation}\label{Mbbdninq}
\|M_b f\|_{L_{\Phi}}\leq C_0 \|b\|_{\ast}\|f\|_{L_{\Phi}}
\end{equation}
holds with constant $C_0$ independent of $f$.
\end{thm}

\begin{lem}\label{lem5.1.com}
Let $\Phi$ be a Young function with $\Phi\in\Delta_2\cap\nabla_2$, $b \in BMO(X)$,
then the inequality
\begin{equation*}\label{eq5.1g}
\|M_{b}f\|_{L_{\Phi}(B(x_0,r))} \lesssim   \frac{\|b\|_{*}}{\Phi^{-1}\big(r^{-Q}\big)}
 \sup_{t>r} \Big(1+\ln \frac{t}{r}\Big) \Phi^{-1}\big(t^{-Q}\big) \|f\|_{L_{\Phi}(B(x_0,t))}
\end{equation*}
holds for any ball $B(x_0,r)$ and for all $f\in L_{\Phi}^{\rm loc}(X)$.
\end{lem}

\begin{proof}
For  $B=B(x_0,r)$, write $f=f_1+f_2$ with
$f_1=f\chi_{_{2kB}}$ and $f_2=f\chi_{_{\dual (2kB)}}$, where $k$ is the constant from the triangle inequality \eqref{triinq}, so that
$$
\left\|M_{b}f\right\|_{L^\Phi(B)} \leq
\left\|M_{b}f_1 \right\|_{L^\Phi(B)}+
\left\|M_{b}f_2 \right\|_{L^\Phi(B)}.
$$

By the boundedness of the operator $M_{b}$  in the space $L_{\Phi}(X)$ provided by
  Theorem \ref{TinOrlicz}, we obtain
\begin{equation}\label{ffiir}
\|M_{b}f_1\|_{L^\Phi(B)}  \leq
\|M_{b}f_1\|_{L^\Phi(X)}
 \lesssim  \|b\|_{*} \, \|f_1\|_{L^\Phi(X)} = \|b\|_{*} \, \|f\|_{L^\Phi(2B)}.
\end{equation}
As we proceed in the proof of Lemma \ref{lem4.2.}, we have for $x \in B$

\begin{align*}
M_{b}(f_2)(x) \lesssim \sup_{t>2kr} \frac{1}{\mu(B(x_0,t))}\int_{B(x_0,t)}|b(y)-b(x)||f(y)|d\mu(y).
\end{align*}

Then
\begin{align*}
& \|M_{b}f_2\|_{L^\Phi(B)} \lesssim
\left\|\sup_{t>r} \frac{1}{\mu(B(x_0,t))}\int_{B(x_0,t)}|b(y)-b(\cdot)||f(y)|d\mu(y)\right\|_{L^\Phi(B)}
\\
&\lesssim J_1+J_2 = \left\|\sup_{t>2r} \frac{1}{\mu(B(x_0,t))}\int_{B(x_0,t)}|b(y)-b_B||f(y)|d\mu(y)\right\|_{L^\Phi(B)}
\\
&\quad +\left\|\sup_{t>r} \frac{1}{\mu(B(x_0,t))}\int_{B(x_0,t)}|b(\cdot)-b_B||f(y)|d\mu(y)\right\|_{L^\Phi(B)}.
\end{align*}

For the term  $J_1$ by \eqref{Qhomogeneous} and \eqref{normofcharac} we obtain
$$
J_1  \thickapprox \frac{1}{\Phi^{-1}\big(r^{-Q}\big)}\sup_{t>r} \frac{1}{\mu(B(x_0,t))}\int_{B(x_0,t)}
|b(y)-b_B||f(y)|d\mu(y)
$$
and split it as follows:
\begin{align*}
J_1 & \lesssim \frac{1}{\Phi^{-1}\big(r^{-Q}\big)}\sup_{t>r}\frac{1}{\mu(B(x_0,t))} \int_{B(x_0,t)}|b(y)-b_{B(x_0,t)}||f(y)|d\mu(y)
\\
&\quad + \frac{1}{\Phi^{-1}\big(r^{-Q}\big)}\sup_{t>r}\frac{1}{\mu(B(x_0,t))}|b_{B(x_0,r)}-b_{B(x_0,t)}|\int_{B(x_0,t)}|f(y)|d\mu(y).
\end{align*}
Applying H\"older's inequality, by Lemmas \ref{lemHold} and \ref{Bmo-orlicz} and \eqref{propBMO} we get
\allowdisplaybreaks
\begin{align*}
J_1 & \lesssim\frac{1}{\Phi^{-1}\big(r^{-Q}\big)}\sup_{t>r}\frac{1}{\mu(B(x_0,t))}\left\|b(\cdot)-b_{B(x_0,t)}\right\|_{L_{\widetilde{\Phi}}(B(x_0,t))} \|f\|_{L^\Phi(B(x_0,t))}
\\
& \quad + \frac{1}{\Phi^{-1}\big(r^{-Q}\big)} \sup_{t>r}\frac{1}{\mu(B(x_0,t))}|b_{B(x_0,r)}-b_{B(x_0,t)}|\mu(B(x_0,t))\Phi^{-1}\left(t^{-Q}\right)\|f\|_{L^\Phi(B(x_0,t))}
\\
& \lesssim \frac{\|b\|_{*}}{\Phi^{-1}\big(r^{-Q}\big)}
\sup_{t>2r}\Phi^{-1}\left(t^{-Q}\right)\Big(1+\ln \frac{t}{r}\Big)
\|f\|_{L^\Phi(B(x_0,t))}.
\end{align*}

For $J_2$  we obtain
\begin{align*}
J_2 & \thickapprox
\left\|b(\cdot)-b_{B}\right\|_{L^\Phi(B)}\sup_{t>r}\frac{1}{\mu(B(x_0,t))}\int_{B(x_0,t)}|f(y)|d\mu(y)
\\
&\lesssim \frac{\|b\|_{*}}{\Phi^{-1}\big(r^{-Q}\big)}\sup_{t>r}\Phi^{-1}\big(t^{-Q}\big)\|f\|_{L^\Phi(B(x_0,t))}
\end{align*}
gathering the estimates for  $J_1$ and $J_2,$ we get
\begin{equation} \label{deckfVgr}
\|M_{b}f_2\|_{L^\Phi(B)}
\lesssim \frac{\|b\|_{*}}{\Phi^{-1}\big(r^{-Q}\big)}
\sup_{t>r}\Phi^{-1}\big(t^{-Q}\big)\Big(1+\ln \frac{t}{r}\Big)
\|f\|_{L^\Phi(B(x_0,t))}.
\end{equation}
By using \eqref{dsaad} we unite \eqref{deckfVgr}  with \eqref{ffiir},
which  completes the proof.
\end{proof}

\begin{thm}\label{thm4.4.maxcom}
Let $\Phi$ be a Young function with $\Phi\in\Delta_2\cap\nabla_2$, $b\in BMO(X)$ and the functions $\varphi_1,\varphi_2$ and $ \Phi$ satisfy the condition
\begin{equation}\label{bounmaxcom}
\sup_{r<t<\infty} \Big(1+\ln \frac{t}{r}\Big)\Phi^{-1}\big(t^{-Q}\big) \es_{t<s<\i}\frac{\varphi_1(s)}{\Phi^{-1}\big(s^{-Q}\big)} \le C \, \varphi_2(r),
\end{equation}
where $C$ does not depend on $r$. Then the operator $M_{b}$ is bounded from $M_{\Phi,\varphi_1}(X)$ to $M_{\Phi,\varphi_2}(X)$.
\end{thm}

\begin{proof}
The proof is similar to the proof of Theorem \ref{thm4.4.max} thanks to Lemma \ref{lem5.1.com}.
\end{proof}

\section{Fractional integral and its commutators in Orlicz spaces}

For a $Q$-homogeneous space $(X,d,\mu)$, let
$$
I_{\alpha} f(x)=\int_{X}\frac{f(y)}{d(x,y)^{Q-\alpha} }d\mu(y),\qquad 0<\a<Q.
$$

For proving our main results, we need the following estimate.

\begin{lem}\label{estRsz}
If $B_0:=B(x_0,r_0)$, then $r_0^{\alpha}\leq C I_{\a} \chi_{B_0}(x)$ for every $x\in B_0$.
\end{lem}
\begin{proof}
If $x,y\in B_0$, then $d(x,y)\leq k(d(x,x_0)+d(y,x_0))<2kr_0$. Since $0<\a<Q$, we get $r_0^{\a-Q}\leq C d(x,y)^{\a-Q}$. Therefore
\begin{align*}
I_{\a} \chi_{B_0}(x)= \int _{B_0} d(x,y)^{\a-Q}d\mu(y) \geq C r_0^{\a-Q} \mu(B_0) = C r_0^{\a}.
\end{align*}
\end{proof}

The known boundedness statement for $I_{\a}$ in Orlicz spaces on spaces of homogeneous type runs as follows.
\begin{thm}\label{bounFrMaxOrl} \cite{Nakai-Hom}
Let $(X,d,\mu)$ be $Q-$homogeneous and $\Phi, \Psi\in\mathcal{Y}$. Assume that there exist constants $A,A^\prime>0$ such that
\begin{equation}\label{cond1}
\int_{r}^{\infty}t^{\alpha-1}\Phi^{-1}\left(t^{-Q}\right)dt\leq A r^\a\Phi^{-1}\left(r^{-Q}\right)\qquad \text{for  }0<r<\infty,
\end{equation}
\begin{equation}\label{cond2}
r^\a \Phi^{-1}\left(r^{-Q}\right)\leq A^\prime \Psi^{-1}\left(r^{-Q}\right)\qquad \text{for  }0<r<\i.
\end{equation}
Then $I_{\a}$ is bounded from $L_{\Phi}(X)$ to $WL_{\Psi}(X)$. Moreover, if $\Phi\in\nabla_2,$ then $I_{\a}$ is bounded from $L_{\Phi}(X)$ to $L_{\Psi}(X)$.
\end{thm}

\begin{thm}\label{bounFrMaxOrlnec}
Let $(X,d,\mu)$ be $Q-$homogeneous and $\Phi, \Psi\in\mathcal{Y}$. Assume that $I_{\a}$ is bounded from $L_{\Phi}(X)$ to $WL_{\Psi}(X)$ then condition
\eqref{cond2} holds.
\end{thm}

\begin{proof}
Let $B_0=B(x_0,r_0)$ and $x\in B_0$. By \eqref{Qhomogeneous} and Lemmas \ref{estRsz} and \ref{normofcharac}, we have
\begin{align*}
r_0^{\alpha}&\lesssim \Psi^{-1}(r_{0}^{-Q})\|I_{\a} \chi_{B_0}\|_{WL_{\Psi}(B_0)} \lesssim \Psi^{-1}(r_{0}^{-Q})\|I_{\a} \chi_{B_0}\|_{WL_{\Psi}}
\\
&\lesssim \Psi^{-1}(r_{0}^{-Q})\|\chi_{B_0}\|_{L_{\Phi}}\lesssim \frac{\Psi^{-1}(r_{0}^{-Q})}{\Phi^{-1}(r_{0}^{-Q})}.
\end{align*}
Since this is true for every $r_0>0$, we are done.
\end{proof}

Combining Theorems \ref{bounFrMaxOrl} and \ref{bounFrMaxOrlnec} we have the following result.
\begin{thm}
Let $(X,d,\mu)$ be $Q-$homogeneous and $\Phi, \Psi\in\mathcal{Y}$. If \eqref{cond1} holds, then the condition \eqref{cond2}
is necessary and sufficient for the boundedness of $I_{\a}$ from $L_{\Phi}(X)$ to $WL_{\Psi}(X)$. Moreover, if $\Phi\in\nabla_2,$ the condition \eqref{cond2} is necessary and sufficient for the boundedness of $I_{\a}$ from $L_{\Phi}(X)$ to $L_{\Psi}(X)$.
\end{thm}

The commutators generated by  $b\in  L^1_{\rm loc}(X)$ and the operator $I_{\a}$ are defined by
\begin{equation*}
[b,I_{\a}]f(x)=\int _{X}\frac{b(x)-b(y)}{d(x,y)^{Q-\a}}f(y)d\mu(y), \qquad 0<\a<Q.
\end{equation*}
The operator $|b,I_{\a}|$ is defined by
\begin{equation*}
|b,I_{\a}|f(x)=\int _{X}\frac{|b(x)-b(y)|}{d(x,y)^{Q-\a}}f(y)d\mu(y), \qquad 0<\a<Q.
\end{equation*}

The following lemma is the analogue of the Hedberg's trick for $[b, I_\a]$.
\begin{lem}\label{AnalgHedCom}
If $(X,d,\mu)$ be $Q-$homogeneous, $0<\alpha<Q$ and $f, b\in  L^1_{\rm loc}(X)$, then for all $x \in X$ and $r>0$ we get
\begin{equation*}
\int_{B(x,r)}\frac{|f(y)|}{d(x,y)^{Q-\alpha}}|b(x)-b(y)|d\mu(y)\lesssim r^{\a} M_{b}f(x).
\end{equation*}
\end{lem}
\begin{proof}
\begin{align*}
  &\int_{B(x,r)}\frac{|f(y)|}{d(x,y)^{Q-\alpha}}|b(x)-b(y)|d\mu(y)  = \sum_{j=0}^{\infty}\int_{2^{-j-1}r \le d(x,y) < 2^{-j}r} \frac{|f(y)|}{d(x,y)^{Q-\alpha}}|b(x)-b(y)|d\mu(y) \\
  & \lesssim \sum_{j=0}^{\infty} (2^{-j}r)^{\a}(2^{-j}r)^{-Q}\int_{d(x,y)< 2^{-j}r} |f(y)||b(x)-b(y)|d\mu(y)  \lesssim r^{\a} M_{b}f(x).
\end{align*}
\end{proof}

\begin{lem}\label{estRszCom}
If $b\in  L^1_{\rm loc}(X)$ and $B_0:=B(x_0,r_0)$, then
$$
r_0^{\alpha}|b(x)-b_{B_0}|\leq C |b,I_{\a}| \chi_{B_0}(x)
$$
for every $x\in B_0$.
\end{lem}
\begin{proof}
The proof is similar to the proof of Theorem \ref{estRsz}.
\end{proof}

\begin{thm} \label{AdamsCommRieszCharOrl}
Let $(X,d,\mu)$ be $Q-$homogeneous, $0<\a<Q$, $b\in BMO(X)$ and $\Phi,\Psi\in\mathcal{Y}$.

$1.~$ If $\Phi\in \nabla_2$ and $\Psi\in \Delta_2$, then the condition
\begin{equation}\label{adRieszCommCharOrl1}
r^{\alpha}\Phi^{-1}\big(r^{-Q}\big) + \int_{r}^{\infty} \Big(1+\ln \frac{t}{r}\Big) \Phi^{-1}\big(t^{-Q}\big)t^{\alpha}\frac{dt}{t} \le C \Psi^{-1}\big(r^{-Q}\big)
\end{equation}
for all $r>0$, where $C>0$ does not depend on $r$, is sufficient for the boundedness of $[b,I_{\a}]$ from $L_{\Phi}(X)$ to $L_{\Psi}(X)$.

$2.~$ If $\Psi\in \Delta_2$, then the condition \eqref{cond2} is necessary for the boundedness of $|b,I_{\a}|$ from $L_{\Phi}(X)$ to $L_{\Psi}(X)$.

$3.~$ Let $\Phi\in \nabla_2$ and $\Psi\in \Delta_2$. If the condition
\begin{equation}\label{adRieszCommCharOrl3}
\int_{r}^{\infty} \Big(1+\ln \frac{t}{r}\Big) \Phi^{-1}\big(t^{-Q}\big) t^{\alpha}\frac{dt}{t} \le C r^{\alpha} \Phi^{-1}\big(r^{-Q}\big)
\end{equation}
holds for all $r>0$, where $C>0$ does not depend on $r$, then the condition \eqref{cond2}
is necessary and sufficient for the boundedness of $|b,I_{\a}|$ from $L_{\Phi}(X)$ to $L_{\Psi}(X)$.
\end{thm}

\begin{proof}
(1) For arbitrary $x_0 \in X$, set $B=B(x_0,r)$ for the ball
centered at $x_0$ and of radius $r$. Write $f=f_1+f_2$ with
$f_1=f\chi_{_{2kB}}$ and $f_2=f\chi_{_{\dual (2kB)}}$, where $k$ is the constant from the triangle inequality \eqref{triinq}.

For $x \in B$ we have
\begin{align*}
|[b,I_{\alpha}]f_2(x)| & \lesssim \int_{X} \frac{|b(y)-b(x)|}{d(x,y)^{Q-\alpha} }|f_2(y)|d\mu(y)
\thickapprox \int_{{\,^{^{\complement}}\!}{(2kB)}} \frac{|b(y)-b(x)|}{d(x_{0},y)^{Q-\alpha} }|f(y)| d\mu(y)
\\
&\lesssim \int_{{\,^{^{\complement}}\!}{(2kB)}} \frac{|b(y)-b_B|}{d(x_{0},y)^{Q-\alpha} }|f(y)| d\mu(y)
+ \int_{{\,^{^{\complement}}\!}{(2kB)}} \frac{|b(x)-b_B|}{d(x_{0},y)^{Q-\alpha} }|f(y)| d\mu(y)\\
&=J_1+J_2(x),
\end{align*}
since $x \in B$ and $y\in \dual (2kB)$ implies
$$\frac{1}{2k}d(x_0,y)\le d(x,y)\le(k+\frac{1}{2})d(x_0,y).$$

Let us estimate $J_1$.
\begin{align*}
J_1&= \int_{\dual
(2kB)}\frac{|b(y)-b_{B}|}{d(x_{0},y)^{Q-\alpha}}|f(y)|d\mu(y)
\thickapprox \int_{\dual
(2kB)}|b(y)-b_{B}||f(y)|\int_{d(x_0,y)}^{\infty}\frac{dt}{t^{Q+1-\a}}d\mu(y)
\\
&\thickapprox  \int_{2kr}^{\infty}\int_{2kr\leq d(x_0,y)\leq t}
|b(y)-b_{B}||f(y)|d\mu(y)\frac{dt}{t^{Q+1-\a}}\\
&\lesssim  \int_{2kr}^{\infty}\int_{B(x_0,t)}
|b(y)-b_{B}||f(y)|d\mu(y)\frac{dt}{t^{Q+1-\a}}.
\end{align*}

Applying H\"older's inequality, by \eqref{2.3}, \eqref{propBMO}, \eqref{Bmo-orlicz} and Lemma \ref{lemHold} we get
\begin{align*}
J_1 & \lesssim \int_{2r}^{\infty}\int_{B(x_0,t)}
|b(y)-b_{B(x_0,t)}||f(y)|d\mu(y)\frac{dt}{t^{Q+1-\a}}
\\
&\quad +  \int_{2r}^{\infty}|b_{B(x_0,r)}-b_{B(x_0,t)}|
\int_{B(x_0,t)} |f(y)|d\mu(y)\frac{dt}{t^{Q+1-\a}}
\\
&\lesssim  \int_{2r}^{\infty}
\left\|b(\cdot)-b_{B(x_0,t)}\right\|_{L_{\widetilde{\Phi}}(B(x_0,t))} \|f\|_{L_\Phi(B(x_0,t))}\frac{dt}{t^{Q+1-\a}}
\\
& \quad +  \int_{2r}^{\infty}|b_{B(x_0,r)}-b_{B(x_0,t)}|
\|f\|_{L_\Phi(B(x_0,t))}\Phi^{-1}\big(\mu(B(x_0,t))^{-1}\big)\frac{dt}{t^{1-\a}}
\\
& \lesssim \|b\|_{*}\,
\int_{2r}^{\infty}\Big(1+\ln \frac{t}{r}\Big)
\|f\|_{L_\Phi(B(x_0,t))}\Phi^{-1}\big(\mu(B(x_0,t))^{-1}\big)\frac{dt}{t^{1-\a}}.
\\
& \lesssim \|b\|_{*}\, \|f\|_{L_{\Phi}}\int_{2r}^{\infty}\Big(1+\ln \frac{t}{r}\Big)
\Phi^{-1}\big(t^{-Q}\big)t^{\alpha}\frac{dt}{t}.
\end{align*}

A geometric observation shows $2kB\subset B(x,\delta)$ for all $x \in B$, where $\delta=(2k+1)kr$. Using Lemma \ref{AnalgHedCom}, we get
\begin{align*}
J_0(x):=|[b,I_{\alpha}]f_1(x)| & \lesssim \int_{2kB} \frac{|b(y)-b(x)|}{d(x,y)^{Q-\alpha} }|f(y)|d\mu(y)
\\
&\lesssim \int_{B(x,\delta)} \frac{|b(y)-b(x)|}{d(x,y)^{Q-\alpha} }|f(y)|d\mu(y) \lesssim r^{\a} M_b f(x).
\end{align*}

Consequently, we have
$$
J_0(x)+J_1 \lesssim \|b\|_{*}r^{\a} M_b f(x)+ \|b\|_{*}\|f\|_{L_{\Phi}}\int_{2r}^{\infty}\Big(1+\ln \frac{t}{r}\Big)
\Phi^{-1}\big(t^{-Q}\big)t^{\alpha}\frac{dt}{t}.
$$

Thus, by \eqref{adRieszCommCharOrl1} we obtain
\begin{align*}
J_0(x)+J_1 & \lesssim \|b\|_{*} \left(M_b f(x)\frac{\Psi^{-1}(r^{-Q})}{\Phi^{-1}(r^{-Q})}+\Psi^{-1}(r^{-Q})\|f\|_{L_{\Phi}}\right).
\end{align*}
Choose $r>0$ so that $\Phi^{-1}(r^{-Q})=\frac{M_b f(x)}{C_0\|b\|_{*}\|f\|_{L_{\Phi}}}$. Then
$$
\frac{\Psi^{-1}(r^{-Q})}{\Phi^{-1}(r^{-Q})}=\frac{(\Psi^{-1}\circ\Phi)(\frac{M_b f(x)}{C_0\|b\|_{*}\|f\|_{L_{\Phi}}})}{\frac{M_b f(x)}{C_0\|b\|_{*}\|f\|_{L_{\Phi}}}}.
$$
Therefore, we get
$$
J_0(x)+J_1 \leq C_1\|b\|_{*} \|f\|_{L_{\Phi}}(\Psi^{-1}\circ\Phi)(\frac{M_b f(x)}{C_0\|b\|_{*}\|f\|_{L_{\Phi}}}).
$$
Let $C_0$ be as in \eqref{Mbbdninq}. Consequently by Theorem \ref{TinOrlicz} we have
\begin{align*}
\int_{B}\Psi\left(\frac{J_0(x)+J_1}{C_1\|b\|_{*}\|f\|_{L_{\Phi}}}\right)d\mu(x)&\leq \int_{B}\Phi\left(\frac{M_b f(x)}{C_0\|b\|_{*}\|f\|_{L_{\Phi}}}\right)d\mu(x)\\
&\leq \int_{X}\Phi\left(\frac{M_b  f(x)}{\|M_b f\|_{L_{\Phi}}}\right)d\mu(x)\leq 1,
\end{align*}
i.e.
\begin{equation}\label{hjkllads}
\|J_0(\cdot)+J_1\|_{L_{\Psi}(B)}\lesssim \|b\|_{*} \|f\|_{L_{\Phi}}.
\end{equation}
In order to estimate $J_2$, by \eqref{Bmo-orlicz}, Lemma \ref{lemHold} and condition \eqref{adRieszCommCharOrl1}, we also get
\begin{align*}
\|J_2\|_{L_{\Psi}(B)} & =
\left\|\int_{\dual (2kB)}
\frac{|b(\cdot)-b_{B}|}{d(x_{0},y)^{Q-\alpha}}|f(y)|d\mu(y)\right\|_{L_{\Psi}(B)}
\\
& \thickapprox
\left\|b(\cdot)-b_{B}\right\|_{L_{\Psi}(B)}\int_{\dual (2kB)}
\frac{|f(y)|}{d(x_{0},y)^{Q-\alpha}}d\mu(y)
\\
&\lesssim \|b\|_{*}\,\frac{1}{\Psi^{-1}\big(r^{-Q}\big)}\int_{\dual (2kB)}
\frac{|f(y)|}{d(x_{0},y)^{Q-\alpha}}d\mu(y)
\\
& \thickapprox
\|b\|_{*}\,\frac{1}{\Psi^{-1}\big(r^{-Q}\big)}\int_{\dual {(2kB)}}|f(y)|\int_{d(x_{0},y)}^{\i}\frac{dt}{t^{Q+1-\alpha}}d\mu(y)\\
&\thickapprox \|b\|_{*}\,\frac{1}{\Psi^{-1}\big(r^{-Q}\big)}\int_{2kr}^{\i}\int_{2kr\leq d(x_{0},y)< t}|f(y)|d\mu(y)\frac{dt}{t^{Q+1-\alpha}}
\\
&\lesssim \|b\|_{*}\,\frac{1}{\Psi^{-1}\big(r^{-Q}\big)}\int_{2r}^{\i}\int_{B(x_0,t) }|f(y)|d\mu(y)\frac{dt}{t^{Q+1-\alpha}}
\\
&\lesssim \|b\|_{*}\,\frac{1}{\Psi^{-1}\big(r^{-Q}\big)}\int_{2r}^{\i}\|f\|_{L_{\Phi}(B(x_0,t))} \Phi^{-1}\big(t^{-Q}\big) t^{\alpha-1} dt
\\
&\lesssim \|b\|_{*}\,
\frac{1}{\Psi^{-1}\big(r^{-Q}\big)}\|f\|_{L_{\Phi}}\,\int_{2r}^{\i}t^{\a}\Phi^{-1}\big(t^{-Q}\big)\frac{dt}{t}\\
&\lesssim \|b\|_{*}\,\|f\|_{L_{\Phi}}.
\end{align*}
Consequently, we have
\begin{equation}\label{hjkll}
\|J_2\|_{L_{\Psi}(B)}\lesssim \|b\|_{*}\,\|f\|_{L_{\Phi}}.
\end{equation}
Combining \eqref{hjkllads} and \eqref{hjkll}, we get
\begin{equation}\label{gfhajqwewr}
\|[b,I_{\a}]f\|_{L_{\Psi}(B)}\lesssim \|b\|_{*}\|f\|_{L_{\Phi}}.
\end{equation}
By taking supremum over $B$ in \eqref{gfhajqwewr}, we get
$$
\|[b,I_{\a}]f\|_{L_{\Psi}}\lesssim \|b\|_{*}\|f\|_{L_{\Phi}},
$$
since the constants in \eqref{gfhajqwewr} do not depend on $x_0$ and $r$.

(2) We shall now prove the second part. Let $B_0=B(x_0,r_0)$ and $x\in B_0$. By Lemmas \ref{estRszCom}, \ref{Bmo-orlicz} and \ref{normofcharac} we have
\begin{align*}
r_0^{\alpha}&\lesssim\frac{\||b,I_{\a}| \chi_{B_0}\|_{L_{\Psi}(B_0)}}{\|b(\cdot)-b_{B_0}\|_{L_{\Psi}(B_0)}} \lesssim \Psi^{-1}(r_0^{-Q})\||b,I_{\a}| \chi_{B_0}\|_{L_{\Psi}(B_0)}
\\
&\lesssim \Psi^{-1}(r_0^{-Q})\||b,I_{\a}| \chi_{B_0}\|_{L_{\Psi}} \lesssim \Psi^{-1}(r_0^{-Q})\|\chi_{B_0}\|_{L_{\Phi}}\lesssim \frac{\Psi^{-1}(r_{0}^{-Q})}{\Phi^{-1}(r_{0}^{-Q})}.
\end{align*}
Since this is true for every $r_0>0$, we are done.

(3) The third statement of the theorem follows from the first and second parts of the theorem.
\end{proof}

\section{Fractional integral and its commutators in generalized Orlicz-Morrey spaces}

The following theorem is one of our main results.
\begin{thm}\label{AdGulRszOrlMorNec}

Let $0<\a<Q$, $\Phi\in\mathcal{Y}$, $\beta\in(0,1)$ and $\eta(t)\equiv\varphi(t)^{\beta}$ and $\Psi(t)\equiv\Phi(t^{1/\beta})$.

$1.~$ If $\Phi\in\nabla_2$ and $\varphi(t)$ satisfies \eqref{bounmax}, then the condition
\begin{equation}\label{eq3.6.V}
t^{\alpha}\varphi(t) + \int_{t}^{\infty} r^{\alpha}\, \varphi(r) \frac{dr}{r} \le C
\varphi(t)^{\beta},
\end{equation}
for all $t>0$, where $C>0$ does not depend $t$, is sufficient for boundedness of $I_{\a}$ from $M_{\Phi,\varphi}(X)$ to $M_{\Psi,\eta}(X)$.

$2.~$ If $\varphi\in{\mathcal{G}}_{\Phi}$, then the condition
\begin{equation}\label{condAdams}
t^{\alpha}\varphi(t)\le C \varphi(t)^{\beta},
\end{equation}
for all $t>0$, where $C>0$ does not depend $t$, is necessary for boundedness of $I_{\a}$ from $M_{\Phi,\varphi}(X)$ to $M_{\Psi,\eta}(X)$.

$3.~$ Let $\Phi\in\nabla_2$. If $\varphi\in{\mathcal{G}}_{\Phi}$
satisfies the regularity condition
\begin{equation}\label{intcondAdams}
\int_{t}^{\infty} r^{\alpha}\, \varphi(r) \frac{dr}{r} \le C t^{\a}\varphi(t),
\end{equation}
for all $t>0$, where $C>0$ does not depend $t$, then the condition \eqref{condAdams}
is necessary and sufficient for boundedness of $I_{\a}$ from $M_{\Phi,\varphi}(X)$ to $M_{\Psi,\eta}(X)$.
\end{thm}

\begin{proof}
\emph{Proof of the first part of the theorem:}

For arbitrary ball $B=B(x,t)$ we represent $f$ as
\begin{equation*}
f=f_1+f_2, \ \quad f_1(y)=f(y)\chi _{B}(y),\quad
 f_2(y)=f(y)\chi_{\dual {(B)}}(y),
\end{equation*}
and have
$$
I_\a f(x)=I_\a f_1(x)+I_\a f_2(x).
$$
For $I_\a f_1(x)$, following Hedberg's trick, we obtain $|I_\a f_1(x)|\leq C_1t^\a Mf(x)$. For $I_\a f_2(x)$ by Lemma \ref{lemHold} we have
\begin{equation*}
\begin{split}
\int_{\dual {(B)}}\frac{|f(y)|}{d(x,y)^{Q-\alpha}}d\mu(y) &
\thickapprox
\int_{\dual {(B)}}|f(y)|\int_{d(x,y)}^{\i}\frac{dr}{r^{Q+1-\alpha}}d\mu(y)
\\
&\thickapprox \int_{t}^{\i}\int_{t\leq d(x,y)< r}|f(y)|d\mu(y)\frac{dr}{r^{Q+1-\alpha}}
\\
&\leq C_2\int_t^\i \Phi^{-1}(r^{-Q})r^{\a-1}\|f\|_{L_{\Phi}(B(x,r))}dr.
\end{split}
\end{equation*}
Consequently we have
\begin{equation*}
\begin{split}
|I_\a f(x)| &
\lesssim
t^\a Mf(x)+\int_t^\i \Phi^{-1}(r^{-Q})r^{\a-1}\|f\|_{L_{\Phi}(B(x,r))}dr
\\
&\lesssim t^\a Mf(x)+\|f\|_{M_{\Phi,\varphi}}\int_t^\i r^\a\varphi(r)\frac{dr}{r}.
\end{split}
\end{equation*}
From \eqref{eq3.6.V} we obtain
\begin{align*}
|I_{\a} f(x)| & \lesssim  \min \{ \varphi(t)^{\beta-1} Mf(x), \varphi(t)^{\beta} \|f\|_{\mathcal{M}^{\Phi,\varphi}}\}
\\
& \lesssim \sup\limits_{s>0} \min \{ s^{\beta-1} Mf(x), s^{\beta} \|f\|_{\mathcal{M}^{\Phi,\varphi}}\}
\\
& = (Mf(x))^{\beta} \, \|f\|_{M_{\Phi,\varphi}}^{1-\beta},
\end{align*}
where we have used that the supremum is achieved when the minimum parts are balanced.
Hence for every $x\in X$ we have
\begin{equation}\label{poiwseest}
|I_{\a} f(x)| \lesssim  (Mf(x))^{\beta} \, \|f\|_{M_{\Phi,\varphi}}^{1-\beta}.
\end{equation}

By using the inequality \eqref{poiwseest} we have
$$
\|I_{\a} f\|_{L_{\Psi}(B)}\lesssim \|(Mf)^{\beta}\|_{L_{\Psi}(B)}\, \|f\|_{M_{\Phi,\varphi}}^{1-\beta}.
$$

Note that from \eqref{orlpr} we get
$$
\int_B \Psi\left(\frac{(Mf(x))^{\beta}}{\|Mf\|_{L_{\Phi}(B)}^{\beta}}\right)d\mu(x)=\int_B \Phi\left(\frac{Mf(x)}{\|Mf\|_{L_{\Phi}(B)}}\right)d\mu(x)\leq 1.
$$
Thus $\|(Mf)^{\beta}\|_{L_{\Psi}(B)}\leq \|Mf\|_{L_{\Phi}(B)}^{\beta}$. Consequently by using this inequality we have
\begin{equation}\label{poiwseestnorm}
\|I_{\a} f\|_{L_{\Psi}(B)}\lesssim \|Mf\|_{L_{\Phi}(B)}^{\beta}\, \|f\|_{M_{\Phi,\varphi}}^{1-\beta}.
\end{equation}
From Theorem \ref{thm4.4.max} and \eqref{poiwseestnorm}, we get
\begin{align*}
\|I_{\a}f\|_{M_{\Psi,\eta}}&=\sup\limits_{x\in X, t>0}\eta(t)^{-1}\Psi^{-1}(t^{-Q})\|I_{\a} f\|_{L_{\Psi}(B)}\\
&\lesssim  \|f\|_{M_{\Phi,\varphi}}^{1-\beta}\, \sup\limits_{x\in X, t>0}\eta(t)^{-1}\Psi^{-1}(t^{-Q})\|Mf\|_{L_{\Phi}(B)}^{\beta}\\
&=\|f\|_{M_{\Phi,\varphi}}^{1-\beta}\, \left(\sup\limits_{x\in X, t>0}\varphi(t)^{-1}\Phi^{-1}(t^{-Q})\|Mf\|_{L_{\Phi}(B)}\right)^{\beta} \\
&\lesssim \|f\|_{M_{\Phi,\varphi}}.
\end{align*}

\emph{Proof of the second part of the theorem:}

Let $B_0=B(x_0,t_0)$ and $x\in B_0$. By Lemma \ref{estRsz} we have $t_0^{\alpha}\leq C I_{\a} \chi_{B_0}(x)$.
Therefore, by \eqref{normofcharac} and Lemma \ref{charOrlMor} we have
\begin{align*}
t_0^{\alpha}&\leq C\Psi^{-1}(\mu(B_0)^{-1})\|I_{\a} \chi_{B_0}\|_{L_{\Psi}(B_0)} \leq C\eta(t_0)\|I_{\a} \chi_{B_0}\|_{M_{\Psi,\eta}}
\\
&\leq C\eta(t_0)\|\chi_{B_0}\|_{M_{\Phi,\varphi}}\leq C\frac{\eta(t_0)}{\varphi(t_0)}= C \varphi(t_0)^{\beta-1}.
\end{align*}
Since this is true for every $t_0>0$, we are done.

The third statement of the theorem follows from first and second parts of the theorem.
\end{proof}

\begin{thm}\label{AdGulRszOrlMorNecW}

Let $0<\a<Q$, $\Phi\in\mathcal{Y}$, $\beta\in(0,1)$ and $\eta(t)\equiv\varphi(t)^{\beta}$ and $\Psi(t)\equiv\Phi(t^{1/\beta})$.

$1.~$ If $\varphi(t)$  satisfies \eqref{bounmax}, then the condition \eqref{eq3.6.V} is sufficient for boundedness of $I_{\a}$ from $M_{\Phi,\varphi}(X)$ to $WM_{\Psi,\eta}(X)$.

$2.~$ If $\varphi\in{\mathcal{G}}_{\Phi}$, then the condition \eqref{condAdams}
is necessary for boundedness of $I_{\a}$ from $M_{\Phi,\varphi}(X)$ to $WM_{\Psi,\eta}(X)$.

$3.~$ If $\varphi\in{\mathcal{G}}_{\Phi}$
satisfies the regularity condition \eqref{intcondAdams}, then the condition \eqref{condAdams}
is necessary and sufficient for boundedness of $I_{\a}$ from $M_{\Phi,\varphi}(X)$ to $WM_{\Psi,\eta}(X)$.
\end{thm}

\begin{proof}
\emph{Proof of the first part of the theorem:}

By using the inequality \eqref{poiwseest} we have
$$
\|I_{\a} f\|_{WL_{\Psi}(B)}\lesssim \|(Mf)^{\beta}\|_{WL_{\Psi}(B)}\, \|f\|_{M_{\Phi,\varphi}}^{1-\beta},
$$
where $B=B(x,t)$.
Note that from \eqref{worlpr} we get
$$
\sup_{t>0}\Psi\left(\frac{t^\beta}{\|Mf\|_{WL_{\Phi}(B)}^{\beta}}\right)d_{(Mf)^{\beta}}(t^\beta)=\sup_{t>0} \Phi\left(\frac{t}{\|Mf\|_{WL_{\Phi}(B)}}\right)d_{Mf}(t)\leq 1.
$$
Thus $\|(Mf)^{\beta}\|_{WL_{\Psi}(B)}\leq \|Mf\|_{WL_{\Phi}(B)}^{\beta}$. Consequently by using this inequality we have
\begin{equation}\label{poiwseestnormw}
\|I_{\a} f\|_{WL_{\Psi}(B)}\lesssim \|Mf\|_{WL_{\Phi}(B)}^{\beta}\, \|f\|_{M_{\Phi,\varphi}}^{1-\beta}.
\end{equation}
From Theorem \ref{thm4.4.max} and \eqref{poiwseestnormw}, we get
\begin{align*}
\|I_{\a}f\|_{WM_{\Psi,\eta}}&=\sup\limits_{x\in X, t>0}\eta(t)^{-1}\Psi^{-1}(t^{-Q})\|I_{\a} f\|_{WL_{\Psi}(B)}\\
&\lesssim \|f\|_{M_{\Phi,\varphi}}^{1-\beta}\, \sup\limits_{x\in X, t>0}\eta(t)^{-1}\Psi^{-1}(t^{-Q})\|Mf\|_{WL_{\Phi}(B)}^{\beta}\, \\
&=\|f\|_{M_{\Phi,\varphi}}^{1-\beta}\, \left(\sup\limits_{x\in X, t>0}\varphi(t)^{-1}\Phi^{-1}(t^{-Q})\|Mf\|_{WL_{\Phi}(B)}\right)^{\beta} \\
&\lesssim \|f\|_{M_{\Phi,\varphi}}.
\end{align*}

\emph{Proof of the second part of the theorem:}
Let $B_0=B(x_0,t_0)$ and $x\in B_0$. By Lemma \ref{estRsz} we have $t_0^{\alpha}\leq C I_{\a} \chi_{B_0}(x)$. Therefore, by \eqref{normofcharac} and Lemma \ref{charOrlMor}
\begin{align*}
t_0^{\alpha}&\leq C\Psi^{-1}(\mu(B_0)^{-1})\|I_{\a} \chi_{B_0}\|_{WL_{\Psi}(B_0)} \leq C\eta(t_0)\|I_{\a} \chi_{B_0}\|_{WM_{\Psi,\eta}}
\\
&\leq C\eta(t_0)\|\chi_{B_0}\|_{M_{\Phi,\varphi}}\leq C\frac{\eta(t_0)}{\varphi(t_0)}= C \varphi(t_0)^{\beta-1}.
\end{align*}
Since this is true for every $t_0>0$, we are done.

The third statement of the theorem follows from first and second parts of the theorem.
\end{proof}

The following theorem is one of our main results.
\begin{thm} \label{comRszNec}

Let $0<\a<n$, $\Phi\in\mathcal{Y}$, $b\in BMO(X)$, $\beta\in(0,1)$ and $\eta(t)\equiv\varphi(t)^{\beta}$ and $\Psi(t)\equiv\Phi(t^{1/\beta})$.

$1.~$ If $\Phi\in \Delta_2\cap\nabla_2$ and $\varphi$ satisfies \eqref{bounmaxcom}, then the condition
\begin{equation}\label{eq3.6.Vxy}
r^{\alpha}\varphi(r) + \int_{r}^{\infty} \Big(1+\ln \frac{t}{r}\Big) \varphi(t)t^{\alpha}\frac{dt}{t} \le C
\varphi(r)^{\beta},
\end{equation}
for all $r>0$, where $C>0$ does not depend on $r$, is sufficient for the boundedness of $[b,I_{\a}]$ from $M_{\Phi,\varphi}(X)$ to $M_{\Psi,\eta}(X)$.

$2.~$ If $\Phi\in \Delta_2$ and $\varphi\in{\mathcal{G}}_{\Phi}$, then the condition \eqref{condAdams}
is necessary for the boundedness of $|b,I_{\a}|$ from $M_{\Phi,\varphi}(X)$ to $M_{\Psi,\eta}(X)$.

$3.~$ Let $\Phi\in \Delta_2\cap\nabla_2$. If $\varphi\in{\mathcal{G}}_{\Phi}$
satisfies the conditions
\begin{equation*}
\sup_{r<t<\infty} \Big(1+\ln \frac{t}{r}\Big) \varphi(t) \le C \, \varphi(r),
\end{equation*}
and
\begin{equation*}
\int_{r}^{\infty} \Big(1+\ln \frac{t}{r}\Big) \varphi(t)t^{\alpha}\frac{dt}{t} \le C
r^{\alpha}\varphi(r),
\end{equation*}
for all $r>0$, where $C>0$ does not depend on $r$, then the condition \eqref{condAdams}
is necessary and sufficient for the boundedness of $|b,I_{\a}|$ from $M_{\Phi,\varphi}(X)$ to $M_{\Psi,\eta}(X)$.
\end{thm}
\begin{proof}
For arbitrary $x_0 \in X$, set $B=B(x_0,r)$ for the ball
centered at $x_0$ and of radius $r$. Write $f=f_1+f_2$ with
$f_1=f\chi_{_{2kB}}$ and $f_2=f\chi_{_{\dual (2kB)}}$, where $k$ is the constant from the triangle inequality \eqref{triinq}.

If we use the same notation and proceed as in the proof of Theorem \ref{AdamsCommRieszCharOrl} for $x \in B$ we have
\begin{align*}
J_0(x)+J_1 &\lesssim \|b\|_{*}r^{\a} M_b f(x)+ \|b\|_{*}\,
\int_{2r}^{\infty}\Big(1+\ln \frac{t}{r}\Big)
\|f\|_{L_\Phi(B(x_0,t))}\Phi^{-1}\big(t^{-Q}\big)\frac{dt}{t^{1-\a}}\\
&\lesssim \|b\|_{*}\Big(r^{\a} M_b f(x)+ \|f\|_{\mathcal{M}^{\Phi,\varphi}}\,
\int_{2r}^{\infty}\Big(1+\ln \frac{t}{r}\Big)
\varphi(t)\frac{dt}{t^{1-\a}}\Big).
\end{align*}

Thus, by \eqref{eq3.6.Vxy} we obtain
\begin{align*}
J_0(x)+J_1 & \lesssim \|b\|_{*} \min \{ \varphi(r)^{\beta-1} M_bf(x), \varphi(r)^{\beta} \|f\|_{\mathcal{M}^{\Phi,\varphi}}\}
\\
& \lesssim \|b\|_{*} \sup\limits_{s>0} \min \{ s^{\beta-1} M_b f(x), s^{\beta} \|f\|_{\mathcal{M}^{\Phi,\varphi}}\}
\\
& = \|b\|_{*} (M_b f(x))^{\beta} \, \|f\|_{\mathcal{M}^{\Phi,\varphi}}^{1-\beta}.
\end{align*}
Consequently for every $x\in B$ we have
\begin{equation}\label{poiwseestpotcom}
J_0(x)+J_1 \lesssim \|b\|_{*} (M_b f(x))^{\beta} \, \|f\|_{\mathcal{M}^{\Phi,\varphi}}^{1-\beta}.
\end{equation}
By using the inequality \eqref{poiwseestpotcom} we have
$$
\|J_0(\cdot)+J_1\|_{L_{\Psi}(B)}\lesssim \|b\|_{*}\|(M_b f)^{\beta}\|_{L_{\Psi}(B)}\, \|f\|_{\mathcal{M}^{\Phi,\varphi}}^{1-\beta}.
$$

Note that from \eqref{orlpr} we get
$$
\int_B \Psi\left(\frac{(M_b f(x))^{\beta}}{\|M_b f\|_{L_{\Phi}(B)}^{\beta}}\right)d\mu(x)=\int_B \Phi\left(\frac{M_b f(x)}{\|M_b f\|_{L_{\Phi}(B)}}\right)d\mu(x)\leq 1.
$$
Thus $\|(M_b f)^{\beta}\|_{L_{\Psi}(B)}\leq \|M_b f\|_{L_{\Phi}(B)}^{\beta}$. Therefore, we have
$$
\|J_0(\cdot)+J_1\|_{L_{\Psi}(B)}\lesssim \|b\|_{*}\|M_b f\|_{L_{\Phi}(B)}^{\beta}\, \|f\|_{\mathcal{M}^{\Phi,\varphi}}^{1-\beta}.
$$

If we also use the same notation and proceed as in the proof of Theorem \ref{AdamsCommRieszCharOrl}, we also get
\begin{align*}
\|J_2\|_{L_{\Psi}(B)} \lesssim \|b\|_{*}\,\frac{1}{\Psi^{-1}\big(r^{-Q}\big)}\int_{2r}^{\i}\|f\|_{L_{\Phi}(B(x_0,t))} \Phi^{-1}\big(t^{-Q}\big) t^{\alpha-1} dt.
\end{align*}
From this estimate and condition \eqref{eq3.6.Vxy} we have
\begin{align*}
\|J_2\|_{L_{\Psi}(B)}
&\lesssim
\frac{\|b\|_{*}}{\Psi^{-1}\big(r^{-Q}\big)}\|f\|_{\mathcal{M}^{\Phi,\varphi}}\,\int_{2r}^{\i}t^{\a}\varphi(t)\frac{dt}{t}\\
&\lesssim
\frac{\|b\|_{*}}{\Psi^{-1}\big(r^{-Q}\big)}\|f\|_{\mathcal{M}^{\Phi,\varphi}}
\varphi(r)^{\beta}.
\end{align*}

Consequently by using Theorem \ref{thm4.4.maxcom}, we get
\begin{align*}
&\|[b,I_{\alpha}]f\|_{\mathcal{M}^{\Psi,\eta}}=\sup\limits_{x_0\in X, r>0}\eta(r)^{-1}\Psi^{-1}(r^{-Q})\|[b,I_{\alpha}]f\|_{L_{\Psi}(B)}\\
&\lesssim\|b\|_{*}\|f\|_{\mathcal{M}^{\Phi,\varphi}}^{1-\beta}\, \left(\sup\limits_{x_0\in X, r>0}\varphi(r)^{-1}\Phi^{-1}(r^{-Q})\|M_b f\|_{L_{\Phi}(B)}\right)^{\beta}+\|b\|_{*}\|f\|_{\mathcal{M}^{\Phi,\varphi}}\\
&\lesssim \|b\|_{*}\|f\|_{\mathcal{M}^{\Phi,\varphi}}.
\end{align*}

We shall now prove the second part. Let $B_0=B(x_0,r_0)$ and $x\in B_0$. By Lemma \ref{estRszCom} we have $r_0^{\alpha}|b(x)-b_{B_0}|\leq C |b,I_{\a}| \chi_{B_0}(x)$. Therefore, by Lemma \ref{Bmo-orlicz} and Lemma \ref{charOrlMor}
\begin{align*}
r_0^{\alpha}&\leq C\frac{\||b,I_{\a}| \chi_{B_0}\|_{L_{\Psi}(B_0)}}{\|b(\cdot)-b_{B_0}\|_{L_{\Psi}(B_0)}} \leq \frac{C}{\|b\|_{*}}\||b,I_{\a}| \chi_{B_0}\|_{L_{\Psi}(B_0)}\Psi^{-1}(r^{-Q})
\\
&\leq \frac{C}{\|b\|_{*}}\eta(r_0)\||b,I_{\a}|\chi_{B_0}\|_{\mathcal{M}^{\Psi,\eta}} \leq C\varphi_2(r_0)\|\chi_{B_0}\|_{\mathcal{M}^{\Phi,\varphi}}\leq C\frac{\eta(r_0)}{\varphi(r_0)}
\leq C\varphi(r_0)^{\beta-1}.
\end{align*}
Since this is true for every $r_0>0$, we are done.

The third statement of the theorem follows from the first and second parts of the theorem.
\end{proof}

%


\begin{thebibliography}{99}

\bibitem{Adams} D.R. Adams,  A note on Riesz potentials, Duke Math. J. 42 (1975), 765-778.

\bibitem{Birnbaum-Orlicz} Z. Birnbaum, W. Orlicz,  Über die verallgemeinerung des begriffes der zueinan-der konjugierten potenzen. Studia Math. \textbf{3}, 1"1¤767 (1931).

\bibitem{BenSharp} C. Bennett and R. Sharpley, \emph{Interpolation of operators}, Academic Press, Boston, 1988.

\bibitem{Cald1} A.P. Calder\'{o}n, Commutators of singular integral operators. Proc. Nat. Acad. Sci. U.S.A. \textbf{53}, 1092-1099 (1965)

\bibitem{Chanillo} S. Chanillo, A note on commutators. Indiana Univ. Math. J. \textbf{31} (1), 7-16 (1982)

\bibitem{ChiFra} F. Chiarenza, M. Frasca, Morrey spaces and Hardy-Littlewood maximal function, Rend. Math. Appl. \textbf{1987}, 7, 7, 273-279.

\bibitem{Cianchi1} A. Cianchi, Strong and weak type inequalities for some classical operators in Orlicz spaces, J. London Math. Soc. {60} (2) (1999), no. 1, 187-202.

\bibitem{CoifWeiss}  R.R. Coifman, G. Weiss, Analyse harmonique non-commutative sur certain espaces
homogenes, in ``Lecture Notes in Math.,'' No. 242, Springer-Verlag, Berlin, 1971

\bibitem{CoifWeiss2} R.R. Coifman, G. Weiss, Extensions of Hardy Spaces and their use in analysis, Bull.
Amer. Math. Soc. 831 (1977), 569-645.

\bibitem{CRW} R.R. Coifman, R. Rochberg, G. Weiss, Factorization theorems for Hardy spaces in several variables.
Ann. of Math.(2) vol \textbf{103} (3), 611-635 (1976)

\bibitem{CLMS} R. Coifman, P. Lions, Y. Meyer, S. Semmes, Compensated compactness and Hardy spaces. J. Math. Pures Appl.\textbf{72}, 247-286 (1993)

%
%

\bibitem{DengHan} D. Deng, Y. Han, \emph{Harmonic analysis on spaces of homogeneous type}, Springer-Verlag, Berlin, 2009. xii+154 pp.

\bibitem{DGS} F. Deringoz, V.S. Guliyev and S.G. Samko,  Boundedness of maximal and singular operators on generalized Orlicz-Morrey spaces, Operator Theory, Operator Algebras and Applications, Series: Operator Theory: Advances and Applications Vol. 242 (2014), 1-24.

\bibitem{DGSJIA} F. Deringoz, V.S. Guliyev, S.G. Hasanov, Characterizations for the Riesz potential and its commutators on generalized Orlicz-Morrey spaces. J. Inequal. Appl. 2016, Paper No. 248, 22 pp.


\bibitem{FuYangYuan} X. Fu, D. Yang, W. Yuan,
Boundedness of multilinear commutators of Calder\'{o}n-Zygmund operators on Orlicz spaces over non-homogeneous spaces. Taiwanese J. Math. \textbf{16}, 2203-2238 (2012)

%
%

\bibitem{GenGogKokKr} I. Genebashvili, A. Gogatishvili, V. Kokilashvili, M. Krbec,
\emph{Weight theory for integral transforms on spaces of homogeneous type}. Longman, Harlow, (1998)


%

\bibitem{GulJIA} V.S. Guliyev, Boundedness of the maximal, potential and singular operators in the generalized Morrey spaces,
J. Inequal. Appl.  2009, Art. ID 503948, 20 pp.

\bibitem{GULAKShIEOT2012} V.S. Guliyev, S.S. Aliyev, T. Karaman, P. S. Shukurov,
Boundedness of sublinear operators and commutators on generalized Morrey Space, Int. Eq. Op. Theory. 71 (3) (2011), 327-355.


\bibitem{GulMusHom}  V.S. Guliyev, R. Ch. Mustafayev, Fractional integrals in spaces of functions defined on spaces of homogeneous type, (Russian) Anal. Math.  24  (1998),  no. 3, 181-200.

\bibitem{GulDerHas} V.S. Guliyev, F. Deringoz, S.G. Hasanov,  Riesz potential and its commutators on Orlicz spaces. J. Inequal. Appl. 2017, Paper No. 75, 18 pp.





\bibitem{IzukiSaw} M. Izuki, Y. Sawano, Characterization of $BMO$ via ball Banach function spaces. Vestn. St.-Peterbg. Univ. Mat. Mekh. Astron. 4(62) (2017), no. 1, 78-86.

%


\bibitem{KrasnRut} M.A. Krasnoselskii and Ya. B. Rutickii, \emph{Convex Functions and Orlicz Spaces}, English translation P. Noordhoff Ltd., Groningen, 1961.




\bibitem{Nakai94} E. Nakai, Hardy--Littlewood maximal operator, singular integral operators and Riesz potentials on generalized Morrey spaces, Math. Nachr. {\bf 166} (1994),  95-103.

\bibitem{Nakai} E. Nakai, On generalized fractional integrals, Taiwanese J. Math. 5 (2001), no. 3, 587-602.

\bibitem{Nakai-Hom}
E. Nakai, On generalized fractional integrals in the Orlicz spaces on spaces of homogeneous type,
Scientiae Mathematicae Japonicae, Vol.54 (2001), 473-487.

\bibitem{Nakai-HomFS} E. Nakai, The Campanato, Morrey and H\"{o}lder spaces on spaces of homogeneous type. Studia Math.  176  (2006),  no. 1, 1-19.

\bibitem{Nakaisurv} E. Nakai, Recent topics of fractional integrals, Sugaku Expositions, American Mathematical Society, Vol.20, No.2 (2007), 215--235.

\bibitem{O'Neil} R. O'Neil, Fractional integration in Orlicz spaces, Trans. Amer. Math. Soc. {115} (1965), 300-328.

\bibitem{Orlicz1}  W. Orlicz, \"{U}ber eine gewisse Klasse von R\"{a}umen vom Typus B, Bull. Acad. Polon. A (1932), 207-220;
reprinted in: Collected Papers, PWN, Warszawa, 1988, 217-230.

\bibitem{Peetre} J. Peetre, On the theory of $\mathcal{L}_{p,\lambda}$, J. Funct. Anal. {4} (1969), 71-87.

\bibitem{RaoRen} M.M. Rao and Z.D. Ren, \emph{Theory of Orlicz Spaces}, M. Dekker, Inc., New York, 1991.

\bibitem{SugTan} S. Sugano, H. Tanaka, Boundedness of fractional integral operators on generalized Morrey spaces. Sci. Math. Jpn.  58  (2003),  no. 3, 531–540.

\bibitem{Torch1} A. Torchinsky, Interpolation of operators and Orlicz classes, Studia Math. 59 (1976), 177-207.

\end{thebibliography}
\end{document}